\documentclass{amsart}

\usepackage{amsmath}
\usepackage{amsthm}
\usepackage{amsfonts}
\usepackage{amssymb}

\newcommand\R{{\mathbb{R}}}
\newcommand\C{{\mathbb{C}}}

\renewcommand\P{{\mathbf{P}}}
\newcommand\E{{\mathbf{E}}}

\renewcommand\Im{{\operatorname{Im}}}
\renewcommand\Re{{\operatorname{Re}}}
\newcommand\eps{{\varepsilon}}

\newcommand\tr{\operatorname{trace}}
\newcommand\dist{\operatorname{dist}}

\theoremstyle{plain}
  \newtheorem{theorem}[subsection]{Theorem}

  \newtheorem{proposition}[subsection]{Proposition}
  
  \newtheorem{lemma}[subsection]{Lemma}

\theoremstyle{remark}
  \newtheorem{remark}[subsection]{Remark}

\theoremstyle{definition}
  \newtheorem{definition}[subsection]{Definition}

\include{psfig}

\begin{document}

\title[Universality up to the edge]{Random matrices: Universality of local eigenvalue statistics up to the edge}

\author{Terence Tao}
\address{Department of Mathematics, UCLA, Los Angeles CA 90095-1555}
\email{tao@math.ucla.edu}
\thanks{T. Tao is is supported by a grant from the MacArthur Foundation, by NSF grant DMS-0649473, and by the NSF Waterman award.}

\author{Van Vu}
\address{Department of Mathematics, Rutgers, Piscataway, NJ 08854}
\email{vanvu@math.rutgers.edu}
\thanks{V. Vu is supported by research grants DMS-0901216 and AFOSAR-FA-9550-09-1-0167.}

\begin{abstract}  This is a continuation of our earlier paper \cite{TVbulk} on the universality of the eigenvalues of Wigner random matrices.  The main new results of this paper are an extension of the results in \cite{TVbulk} from the bulk of the spectrum up to the edge.  In particular, we prove a variant of the universality results of Soshnikov \cite{Sos1} for the largest eigenvalues, assuming moment conditions rather than symmetry conditions.  The main new technical observation is that there is a significant bias in the Cauchy interlacing law near the edge of the spectrum which allows one to continue ensuring the delocalization of eigenvectors.
\end{abstract}

\maketitle

\setcounter{tocdepth}{2}

\section{Introduction}

In our recent paper \cite{TVbulk}, a universality result (the Four Moment Theorem) was established for the eigenvalue spacings in the bulk of the spectrum of random Hermitian matrices.  (See \cite{Deisur} for an extended discussion of the universality phenomenon, and \cite{TVbulk} for further references on universality results in the context of Wigner Hermitian matrices.)

The main purpose of this paper is to extend this universality result to the edge of the spectrum as well.

\subsection{Universality in the bulk}

To recall the Four Moment Theorem, we need some notation.

\begin{definition}[Condition {\bf C0}]\label{co-def}  A random Hermitian matrix $M_n = (\zeta_{ij})_{1 \leq i, j \leq n}$ is said to obey \emph{condition {\bf C0}} if
\begin{itemize}
\item The $\zeta_{ij}$ are independent (but not necessarily identically distributed) for $1 \leq i \leq j \leq n$.  For $1 \leq i < j \leq n$, they have mean zero and variance $1$; for $i=j$, they have mean zero and variance $c$ for some fixed $c>0$ independent of $n$.

\item  (Uniform exponential decay) There exist constants $C, C' > 0$ such that
\begin{equation}\label{ued}
\P( |\zeta_{ij}| \ge t^C) \le \exp(- t)
\end{equation}
for all $t \ge C'$ and $1 \leq i,j \leq n$.
\end{itemize}
\end{definition}

Examples of random Hermitian matrices obeying Condition {\bf C0} include the GUE and GOE ensembles, or the random symmetric Bernoulli ensemble in which each of the $\zeta_{ij}$ are equal to $\pm 1$ with equal probability $1/2$.  In GOE one has $c=2$, but in the other two cases one has $c=1$.  The arguments in the previous paper \cite{TVbulk} were largely phrased for the case $c=1$, but it is not difficult to see that the arguments extend without difficulty to other values of $c$ (the main point being that a modification of the variance of a single entry of a row vector does not significantly affect the Talagrand concentration inequality, \cite[Lemma 43]{TVbulk}, or Lemma \ref{lemma:projection} below.).

Given an $n \times n$ Hermitian matrix $A$, we denote its $n$ eigenvalues as
$$ \lambda_1(A) \leq \ldots \leq \lambda_n(A),$$
and write $\lambda(A) := (\lambda_1(A),\ldots,\lambda_n(A))$.  We also let $u_1(A),\ldots,u_n(A) \in \C^n$ be an orthonormal basis of eigenvectors of $A$ with $A u_i(A) = \lambda_i(A) u_i(A)$; these eigenvectors $u_i(A)$ are only determined up to a complex phase even when the eigenvalues are simple, but this ambiguity will not cause a difficulty in our results as we will only be interested in the \emph{magnitude} $|u_i(A)^* X|$ of various inner products $u_i(A)^* X$ of $u_i(A)$ with other vectors $X$.

It will be convenient to introduce the following notation for frequent events depending on $n$, in increasing order of likelihood:

\begin{definition}[Frequent events]\label{freq-def}  Let $E$ be an event depending on $n$.
\begin{itemize}
\item $E$ holds \emph{asymptotically almost surely} if\footnote{See Section \ref{notation-sec} for our conventions on asymptotic notation.} $\P(E) = 1-o(1)$.
\item $E$ holds \emph{with high probability} if $\P(E) \geq 1-O(n^{-c})$ for some constant $c>0$.
\item $E$ holds \emph{with overwhelming probability} if $\P(E) \geq 1-O_C(n^{-C})$ for \emph{every} constant $C>0$ (or equivalently, that $\P(E) \geq 1 - \exp(-\omega(\log n))$).
\item $E$ holds \emph{almost surely} if $\P(E)=1$.  
\end{itemize}
\end{definition}

\begin{definition}[Moment matching]\label{def:match}
We say that two complex random variables $\zeta$ and $\zeta'$ \emph{match to order} $k$ if
$$ \E \Re(\zeta)^m \Im(\zeta)^l = \E \Re(\zeta')^m \Im(\zeta')^l$$
for all $m, l \ge 0$ such that $m+l  \le k$. 
\end{definition}

The first main result \cite{TVbulk} can now be stated as follows:

\begin{theorem}[Four Moment Theorem]\label{theorem:main}\cite[Theorem 15]{TVbulk} There is a small positive constant $c_0$ such that for every $0 < \eps < 1$ and $k \geq 1$ the following holds.
 Let $M_n = (\zeta_{ij})_{1 \leq i,j \leq n}$ and $M'_n = (\zeta'_{ij})_{1 \leq i,j \leq n}$ be
 two random matrices satisfying {\bf C0}. Assume furthermore that for any $1 \le  i<j \le n$, $\zeta_{ij}$ and
 $\zeta'_{ij}$  match to order $4$
  and for any $1 \le i \le n$, $\zeta_{ii}$ and $\zeta'_{ii}$ match  to order $2$.  Set $A_n := \sqrt{n} M_n$ and $A'_n := \sqrt{n} M'_n$,
  and let $G: \R^k \to \R$ be a smooth function obeying the derivative bounds
\begin{equation}\label{G-deriv}
|\nabla^j G(x)| \leq n^{c_0}
\end{equation}
for all $0 \leq j \leq 5$ and $x \in \R^k$.
 Then for any $\eps n \le i_1 < i_2 \dots < i_k \le (1-\eps)n$, and for $n$ sufficiently large depending on $\eps,k$ (and the constants $C, C'$ in Definition \ref{co-def}) we have
\begin{equation} \label{eqn:approximation}
 |\E ( G(\lambda_{i_1}(A_n), \dots, \lambda_{i_k}(A_n))) -
 \E ( G(\lambda_{i_1}(A'_n), \dots, \lambda_{i_k}(A'_n)))| \le n^{-c_0}.
\end{equation}
If $\zeta_{ij}$ and $\zeta'_{ij}$ only match to order $3$ rather than $4$, 
then there is a positive  constant $C$ independent of $c_{0}$ such that 
the conclusion \eqref{eqn:approximation} still holds provided that one strengthens \eqref{G-deriv} to
$$
|\nabla^j G(x)| \leq n^{-C j c_0}
$$
for all $0 \leq j \leq 5$ and $x \in \R^k$.
\end{theorem}

Informally, this theorem asserts that the distribution of any bounded number of eigenvalues in the bulk of the spectrum of a random Hermitian matrix obeying condition {\bf C0} depends only on the first four moments of the coefficients.

There is also a useful companion result to Theorem \ref{theorem:main}, which is used both in the proof of that theorem, and in several of its applications:

\begin{theorem}[Lower tail estimates]\label{ltail}\cite[Theorem 17]{TVbulk} Let $0 < \eps < 1$ be a constant, and let $M_n$ be a random matrix obeying Condition {\bf C0}.  Set $A_n := \sqrt{n} M_n$.  Then for every $c_0 > 0$, and for $n$ sufficiently large depending on $\eps$, $c_0$ and the constants $C, C'$ in Definition \ref{co-def}, and for each $\eps n \leq i \leq (1-\eps) n$, one has $\lambda_{i+1}(A_n) - \lambda_i(A_n) \ge n^{-c_0}$
 with high probability. In fact, one has
$$ \P( \lambda_{i+1}(A_n) - \lambda_i(A_n) \leq n^{-c_0} ) \leq n^{-c_1}$$
for some $c_1 > 0$ depending on $c_0$ (and independent of $\eps$).
\end{theorem}

Theorem \ref{theorem:main} (and to a lesser extent, Theorem \ref{ltail}) can be used to extend the range of applicability for various results on eigenvalue statistics in the bulk for Hermitian or symmetric matrices, for instance in extending results for special ensembles such as GUE or GOE (or ensembles obeying some regularity or divisibility conditions) to more general classes of matrices.  See \cite{TVbulk}, \cite{ERSTVY}, \cite{ESY4} for some examples of this type of extension.  We also remark that a level repulsion estimate which has a similar spirit to Theorem \ref{ltail} was established in \cite[Theorem 3.5]{ESY3}, although the latter result establishes repulsion of eigenvalues in a fixed small interval $I$, rather than at a fixed index $i$ of the sequence of eigenvalues, and does not seem to be directly substitutable for Theorem \ref{ltail} in the arguments of this paper.

The results of Theorem \ref{theorem:main} and Theorem \ref{ltail} only control eigenvalues $\lambda_i(A_n)$ in the \emph{bulk} region $\eps n \leq i \leq (1-\eps) n$ for some fixed $\eps > 0$ (independent of $n$).  The reason for this restriction was technical, and originated from the use of the following two related results (which are variants of 
  previous results of Erd\H{o}s, Schlein, and Yau\cite{ESY1, ESY2, ESY3}), whose proof relied on the assumption that one was in the bulk:

\begin{theorem}[Concentration for ESD]\label{sdb}\cite[Theorem 56]{TVbulk}  For any $\eps, \delta > 0$ and any random Hermitian matrix $M_n = (\zeta_{ij})_{1 \leq i,j \leq n}$ whose upper-triangular entries are independent with mean zero and variance $1$, and such that $|\zeta_{ij}| \leq K$ almost surely for all $i,j$ and some $1 \leq K \leq n^{1/2-\eps}$, and any interval $I$ in $[-2+\eps, 2-\eps]$ of width $|I| \geq \frac{K^2 \log^{20} n}{n}$, the number of eigenvalues $N_I$ of $W_n := \frac{1}{\sqrt{n}} M_n$ in $I$ obeys the concentration estimate
$$ |N_I - n \int_I \rho_{sc}(x)\ dx| \ll \delta n |I|$$
with overwhelming probability, where $\rho_{sc}$ is the semicircular distribution
\begin{equation}\label{semi}
 \rho_{sc} (x):= \begin{cases} \frac{1}{2\pi} \sqrt {4-x^2}, &|x| \le 2 \\ 0,
&|x| > 2. \end{cases} 
\end{equation}
In particular, $N_I = \Theta_\eps(n |I|)$ with overwhelming probability.
\end{theorem}

\begin{proposition}[Delocalization of eigenvectors]\label{deloc}\cite[Proposition 58]{TVbulk} Let $\eps, M_n, W_n,\zeta_{ij}, K$  be as in Theorem \ref{sdb}.  Then for any $1 \leq i \leq n$ with $\lambda_i(W_n) \in [-2+\eps,2-\eps]$, if $u_i(W_n)$ denotes a unit eigenvector corresponding to $\lambda_i(W_n)$, then with overwhelming probability each coordinate of $u_i(M_n)$ is $O_{\eps}( \frac{K^2 \log^{20} n}{n^{1/2}} )$.
\end{proposition}

In the bulk region $[-2+\eps,2-\eps]$, the semicircular function $\rho_{sc}$ is bounded away from zero.  Thus, Theorem \ref{sdb} ensures that the eigenvalues of $W_n$ in the bulk tend to have a mean spacing of $\Theta_\eps(1/n)$ on the average.  Applying the Cauchy interlacing law
\begin{equation}\label{interlace}
 \lambda_i(W_n) \leq \lambda_i(W_{n-1}) \leq \lambda_{i+1}(W_n),
\end{equation}
where $W_{n-1}$ is an $n-1 \times n-1$ minor of $W_n$, this implies that the bulk eigenvalues of $W_{n-1}$ are within $\Theta_\eps(1/n)$ of the corresponding eigenvalues of $W_n$ on the average.  Using linear algebra to express the coordinates of the eigenvector $u_i(M_n)$ in terms of $W_n$ and a minor $W_{n-1}$ (see Lemma \ref{lemma:firstcoordinate} below), and using some concentration of measure results concerning the projection of a random vector to a subspace (see Lemma \ref{lemma:projection}), we eventually obtain Proposition \ref{deloc}.

\subsection{Universality up to the edge}

The main results of this paper are that the above four theorems can be extended to the edge of the spectrum (thus effectively sending $\eps$ to zero).  Let us now state these results more precisely.  Firstly, we have the following extension of Theorem \ref{sdb}:

\begin{theorem}[Concentration for ESD up to edge]\label{sdb-2}   Consider a  random Hermitian matrix $M_n = (\zeta_{ij})_{1 \leq i,j \leq n}$ whose upper-triangular entries are independent with mean zero and variance $1$, and such that $|\zeta_{ij}| \leq K$ almost surely for all $i,j$ and some $K \geq 1$.  Let $0 < \delta < 1/2$ be a quantity which can depend on $n$, and let $I$ be an interval such that
$$ |I| \geq \frac{K^2 \log^4 n}{n \delta^{10}}.$$
We also make the mild assumption $K = o( n^{1/2} \delta^2 )$.
Then the number of eigenvalues $N_I$ of $W_n := \frac{1}{\sqrt{n}} M_n$ in $I$ obeys the concentration estimate
$$ |N_I - n \int_I \rho_{sc}(x)\ dx| \ll \delta n |I|$$
with overwhelming probability.  
\end{theorem}

\begin{remark} The powers of $K$, $\delta$ and $\log n$ here are probably not best possible, but this will not be relevant for our purposes.  In our applications $K$ will be a power of $\log n$, and $\delta$ will be a negative power of $\log n$.  (This allows the error term $O(\delta n |I|)$ in the above estimate for $N_I$ to exceed the main term $n \int_I \rho_{sc}(x)\ dx$ when one is very near the edge, but this will not impact our arguments.)
\end{remark}

We prove this theorem in Section \ref{esd-asym}, using the same (standard) Stieltjes transform method that was used to prove Theorem \ref{sdb} in \cite{TVbulk} (see also \cite{ESY3}), with a somewhat more careful analysis.   We next use it to obtain the following extension of Proposition \ref{deloc}:

\begin{proposition}[Delocalization of eigenvectors up to the edge]\label{deloc2} Let $M_n$ be a random matrix obeying condition {\bf C0}.  Then with overwhelming probability, every unit eigenvector $u_i(M_n)$ of $M_n$ has coefficients at most $n^{-1/2} \log^{O(1)} n$, thus
$$ \sup_{1 \leq i,j \leq n} |u_i(M_n)^* e_j| \ll n^{-1/2} \log^{O(1)} n$$
where $e_1,\ldots,e_n$ is the standard basis.
\end{proposition}

The deduction of Proposition \ref{deloc2} from Theorem \ref{sdb-2} differs significantly 
from the analogous deduction of Proposition \ref{deloc} in Theorem \ref{sdb} in \cite{TVbulk}.
The main difference is that in the current case 
$\rho_{sc}$ is no longer bounded away from zero, which causes the average eigenvalue spacing between $\lambda_i(W_n)$ and $\lambda_{i+1}(W_n)$ to be considerably larger than $1/n$. For  instance, the gap between the second largest eigenvalue $\lambda_{n-1}(W_n)$ and the largest eigenvalue $\lambda_n(W_n)$ is typically of size $n^{-2/3}$.

The key new ingredient that help us to deal with this problem is the following observation:
 the Cauchy interlacing law \eqref{interlace}, when applied to the eigenvalues of the edge, is strongly {\it bias}. 
 In particular, 
 the gap between $\lambda_i(W_{n-1})$ and $\lambda_{i}(W_n)$ is significantly smaller than the gap between  $\lambda_i(W_{n-1})$ and $\lambda_{i+1}(W_n)$.
 We can show that (with high probability), the first gap is of order $n^{-1+o(1)}$ while the second can be as large as $n^{-2/3}$
 (and similarly for the gap between $\lambda_{i+1} (W_n)$ and $\lambda_i(W_{n-1})$ when $n/2 \leq i \leq n$). 
 This new ingredient will be sufficient  to recover Proposition \ref{deloc2}; see Section \ref{deloc-sec}, where the above proposition is proved.

Using Theorem \ref{sdb-2} and Proposition \ref{deloc2}, one can continue the arguments from \cite{TVbulk} to establish the following extensions of Theorem \ref{theorem:main} and Theorem \ref{ltail}:

\begin{theorem}[Four Moment Theorem up to the edge]\label{theorem:main2} There is a small positive constant $c_0$ such that for every $k \geq 1$ the following holds.
 Let $M_n = (\zeta_{ij})_{1 \leq i,j \leq n}$ and $M'_n = (\zeta'_{ij})_{1 \leq i,j \leq n}$ be
 two random matrices satisfying {\bf C0}. Assume furthermore that for any $1 \le  i<j \le n$, $\zeta_{ij}$ and
 $\zeta'_{ij}$  match to order $4$
  and for any $1 \le i \le n$, $\zeta_{ii}$ and $\zeta'_{ii}$ match  to order $2$.  Set $A_n := \sqrt{n} M_n$ and $A'_n := \sqrt{n} M'_n$,
  and let $G: \R^k \to \R$ be a smooth function obeying the derivative bounds \eqref{G-deriv}
for all $0 \leq j \leq 5$ and $x \in \R^k$.
 Then for any $1 \le i_1 < i_2 \dots < i_k \le n$, and for $n$ sufficiently large depending on $k$ (and the constants $C, C'$ in Definition \ref{co-def}) we have \eqref{eqn:approximation}.
If $\zeta_{ij}$ and $\zeta'_{ij}$ only match to order $3$ rather than $4$, 
then there is a positive  constant $C$ independent of $c_{0}$ such that 
the conclusion \eqref{eqn:approximation} still holds provided that one strengthens \eqref{G-deriv} to
\begin{equation}\label{g2}
|\nabla^j G(x)| \leq n^{-C j c_0}
\end{equation}
for all $0 \leq j \leq 5$ and $x \in \R^k$.
\end{theorem}

\begin{theorem}[Lower tail estimates up to the edge]\label{ltail2}  Let $M_n$ be a random matrix obeying Condition {\bf C0}.  Set $A_n := \sqrt{n} M_n$.  Then for every $c_0 > 0$, and for $n$ sufficiently large depending on $\eps$, $c_0$ and the constants $C, C'$ in Definition \ref{co-def}, and for each $1 \leq i \leq n$, one has $\lambda_{i+1}(A_n) - \lambda_i(A_n) \ge n^{-c_0}$ with high probability, uniformly in $i$.
\end{theorem}

The novelty  here is that we  have no assumption on the indices $i_j$ and $i$. 
We present the proof of these Theorems in Sections \ref{ltail-sec}, \ref{fmt-sec}, following the arguments in \cite{TVbulk} closely.

\subsection{Applications}

As Theorems \ref{theorem:main2}, \ref{ltail2} extend Theorems \ref{theorem:main}, \ref{ltail}, all the applications of the latter theorems in \cite{TVbulk} (concerning the bulk of the spectrum) can also be viewed as applications of these theorems.  But because these results extend all the way to the edge, we can now obtain some results on the edge of the spectrum as well.  For instance, we can prove

\begin{theorem} \label{theorem:Sos} Let $k$ be a fixed integer and $M_n$ be a matrix obeying Condition {\bf C0}, and suppose that the real and imaginary part of the atom variables have the same covariance matrix as the GUE ensemble (i.e. both components have variance $1/2$, and have covariance $0$). Assume furthermore that all third moments of the atom variables vanish.   Set $W_n := \frac{1}{\sqrt{n}} M_n$. Then the joint distribution of the
$k$ dimensional random vector 
\begin{equation}\label{job}
\Big( (\lambda_n(W_n) -2)n^{2/3}, \dots,
(\lambda_{n-k+1}(W_n)-2) n^{2/3} \Big) 
\end{equation}
has a weak limit as $n \rightarrow \infty$, which coincides with that in the GUE case (in particular, the limit for $k=1$ is the GUE Tracy-Widom distribution\cite{tracy0}, and for higher $k$ is controlled by the Airy kernel \cite{forrester}). The result also holds for the smallest eigenvalues $\lambda_1, \dots, \lambda_k$, with the offset $-2$ replaced by $+2$.

If the atom variables have the same covariance matrix as the GOE ensemble (i.e. they are real with variance $1$ off the diagonal, and $2$ on the diagonal), instead of the GUE ensemble, then the same conclusion applies but with the GUE distribution replaced of course by the GOE distribution (see \cite{tracy} for the $k=1$ case).
\end{theorem}

This result was previously established by Soshnikov \cite{Sos1} (see also \cite{sinai1}, \cite{sinai2}) in the case when $M_n$ is a Wigner Hermitian matrix (i.e. the off-diagonal entries are iid, and the matrix matches GUE to second order at least) with symmetric distribution (which implies, but is stronger than, matching to third order).  For some additional partial results in the non-symmetric case see \cite{peches}, \cite{peches2}.  The exponential decay condition in Soshnikov's result has been lowered to a finite number of moments; see \cite{Ruz}, \cite{Khor}.  It is reasonable to conjecture that the exponential decay conditions in this current paper can similarly be lowered, but we will not pursue this issue here.  It also seems plausible that the third moment matching conditions could be dropped, though this is barely beyond the reach of the current method\footnote{\emph{Note added in proof}: the third moment condition has recently been dropped in \cite{Joh2}, by combining the four moment theorem here with a new proof of universality for the distribution of the largest eigenvalue for gauss divisible matrices.}.

\begin{proof}  We just prove the claim for the largest $k$ eigenvalues and for GUE, as the claim for the smallest $k$ and/or GOE is similar.

Set $A_n := \sqrt{n} M_n$.  It suffices to show that for every smooth function $G: \R^k \to \R$, that the expectation
\begin{equation}\label{egg}
 \E G( (\lambda_n(A_n)-2n)/n^{1/3}, \ldots, (\lambda_{n-k+1}(A_n)-2n)/n^{1/3} ) 
\end{equation}
only changes by $o(1)$ when the matrix $M_n$ is replaced with GUE.  But this follows from the final conclusion of Theorem \ref{theorem:main2}, thanks to the extra factor $n^{-1/3}$. 
\end{proof}

\begin{remark} Notice that there is some room to spare in this argument, as the $n^{-1/3}$ gain in \eqref{egg} is far more than is needed for \eqref{g2}.  Because of this, one can obtain similar universality results for suitably normalised eigenvalues $\lambda_i(A_n)$ with $i \leq n^{1-\eps}$ or $i \geq n - n^{1-\eps}$ for any $\eps > 0$ (where the normalisation factor for $\lambda_i(A_n)$ is now $n^{2/3} \min(i,n-i)^{1/3}$, and the offset $-2$ is replaced by $-t$, where $\int_{-2}^t \rho_{sc}(x)\ dx = \frac{i}{n}$).  We omit the details.
\end{remark} 

\begin{remark} In analogy with \cite{ERSTVY}, one should be able to drop the third moment condition in Theorem \ref{theorem:Sos} if one can control the distribution of the largest (or smallest) eigenvalues from random matrices obtained from a suitable Ornstein-Uhlenbeck process, as in \cite{ERSY2}.
\end{remark}

\subsection{Notation}\label{notation-sec}

We consider $n$ as an asymptotic parameter tending
to infinity.  We use $X \ll Y$, $Y \gg X$, $Y = \Omega(X)$, or $X =
O(Y)$ to denote the bound $X \leq CY$ for all sufficiently large $n$
and for some constant  $C$. Notations such as
 $X \ll_k Y, X= O_k(Y)$ mean that the hidden constant $C$ depend on
 another constant $k$. $X=o(Y)$ or $Y= \omega(X)$ means that
 $X/Y \rightarrow 0$ as $n \rightarrow \infty$; the rate of decay here will be allowed to depend on other parameters.  We write $X = \Theta(Y)$ for $Y \ll X \ll Y$.

We view vectors $x \in \C^n$ as column vectors. The Euclidean norm of a vector $x \in \C^n$ is defined as $\|x\| := (x^* x)^{1/2}$. 

Eigenvalues are always ordered in increasing order, thus for instance $\lambda_n(A_n$) is the largest eigenvalue of a Hermitian matrix $A_n$, and $\lambda_1(A_n)$ is the smallest.

\subsection{Acknowledgements} 

The authors thank the anonymous referee for helpful comments and references, and Horng-Tzer Yau for additional references.

\section{General tools}

In this section we record some general tools (proven in \cite{TVbulk}) which we will use repeatedly in the sequel.  We begin with a very useful concentration of measure result that describes the projection of a random vector to a subspace.

\begin{lemma}[Projection Lemma]\label{lemma:projection}   
  Let $X=(\xi_1, \dots, \xi_n) \in \C^n$ be a
random vector whose entries are independent with mean zero, variance $1$, 
and are bounded in magnitude by $K$ almost surely for some $K $, where $K \ge 10( \E |\xi|^{4} +1)$. Let $H$ be a subspace of dimension $d$ and
$\pi_H$ the orthogonal projection onto $H$. Then
$$\P (|\|\pi_H (X)\| - \sqrt d| \ge t) \le 10 \exp(-
\frac{t^{2}}{10K^2}). $$
In particular, one has
$$ \| \pi_H(X)\| = \sqrt{d} + O( K \log n )$$
with overwhelming probability.

The same conclusion holds (with $10$ replaced by another explicit constant) if one of the entries $\xi_j$ of $X$ is assumed to have variance $c$ instead of $1$, for some absolute constant $c>0$.
\end{lemma}

\begin{proof} See \cite[Lemma 40]{TVbulk}.  (The main tool in the proof is Talagrand's concentration inequality.)  It is clear from the triangle inequality that the modification of variance in a single entry does not significantly affect the conclusion except for constants.
\end{proof}

Next, we record a crude but useful upper bound on the number of eigenvalues in a short interval.

\begin{lemma}[Upper bound on ESD]\label{lemma:smallsc}  
Consider a  random Hermitian matrix $M_n = (\zeta_{ij})_{1 \leq i,j \leq n}$ whose upper-triangular entries are independent with mean zero and variance $1$ (with variance $c$ on the diagonal for some absolute constant $c>0$), and such that $|\zeta_{ij}| \leq K$ almost surely for all $i,j$ and some $K \geq 1$.  Set $W_n := \frac{1}{\sqrt{n}} M_n$. 
Then for any interval $I \subset \R$  with $|I| \geq \frac{K^2 \log^2 n}{n}$, 
$$ N_I \ll n |I|$$
with overwhelming probability, where $N_I$ is the number of eigenvalues of $W_n$ in $I$.
\end{lemma}

\begin{proof}  See \cite[Proposition 62]{TVbulk}.  (The main tools in the proof are the Stieltjes transform method, Lemma \ref{cool} below, and Lemma \ref{lemma:projection}.)  Again, the generalisation to variances other than $1$ on the diagonal do not cause significant changes to the argument.
\end{proof}

Finally, we recall a Berry-Ess\'een type theorem:

\begin{theorem}[Tail bounds for complex random walks]\label{bes}  Let $1 \leq N \leq n$ be integers, and let $A = (a_{i,j})_{1 \leq i \leq N; 1 \leq j \leq n}$ be an $N \times n$ complex matrix whose $N$ rows are orthonormal in $\C^n$, and obeying the incompressibility condition
\begin{equation}\label{summer}
 \sup_{1 \leq i \leq N; 1 \leq j \leq n} |a_{i,j}| \leq \sigma
\end{equation}
for some $\sigma > 0$.  Let $\zeta_1,\ldots,\zeta_n$ be independent complex random variables with mean zero, variance $\E |\zeta_j|^2$ equal to $1$, and obeying $\E |\zeta_{i} |^{3} \le C$ for some $C \geq 1$.  For each $1 \leq i \leq N$, let $S_i$ be the complex random variable
$$ S_i := \sum_{j=1}^n a_{i,j} \zeta_j$$
and let $\vec S$ be the $\C^N$-valued random variable with coefficients $S_1,\ldots,S_N$.
\begin{itemize}
\item (Upper tail bound on $S_i$)  For $t \geq 1$, we have $\P( |S_i| \geq t ) \ll \exp(-ct^2) + C \sigma$ for some absolute constant $c>0$.
\item (Lower tail bound on $\vec S$)  For any $t \leq \sqrt{N}$, one has $\P( |\vec S| \leq t) \ll  O( t/\sqrt{N} )^{\lfloor N/4\rfloor} + C N^4 t^{-3} \sigma$.
\end{itemize}
The same claim holds if one of the $\zeta_i$ is assumed to have variance $c$ instead of $1$ for some absolute constant $c>0$.
\end{theorem}

\begin{proof} See \cite[Theorem 41]{TVbulk}.  Again, the modification of the variance on a single entry can be easily seen to have no substantial effect on the conclusion.
\end{proof}
\section{Asymptotics for the ESD}\label{esd-asym}

In this section we prove Theorem \ref{sdb-2}, using the Stieltjes transform method (see \cite{BS} for a general discussion of this method).  We may assume throughout that $n$ is large, since the claim is vacuous for $n$ small. 

It is known by the moment method (see e.g. \cite{BS} or \cite{baiyin2}) that with overwhelming probability, all eigenvalues of $W_n$ have magnitude at most $2+o(1)$.  Because of this, we may restrict attention to the case when $I$ lies in interval $[-3,3]$ (say).

We recall the \emph{Stieltjes transform} $s_n(z)$ of a Hermitian matrix $W_n$, defined for complex $z$ by the formula
\begin{equation}\label{sndef}
 s_n(z) := \frac{1}{n} \sum_{i=1}^n \frac{1}{\lambda_i(W_n)-z}.
\end{equation}
We also introduce the semicircular counterpart
$$ s(z) := \int_{-2}^2 \frac{1}{x-z} \rho_{sc}(x)\ dx$$
which by a standard contour integral computation can be given explicitly as
\begin{equation}\label{explicit}
s(z) = \frac{1}{2} ( - z + \sqrt{z^2-4} )
\end{equation}
where we use the branch of the square root of $z^2-4$ with cut at $[-2,2]$ which is asymptotic to $z$ at infinity.

It is well known that one can control the empirical spectral distribution $N_I$ via the Stieltjes transform.  We will use the following formalization of this principle:

\begin{lemma}[Control of Stieltjes transform implies control on ESD]\label{lemma:S-transform}
There is a positive constant $C$ such that the following holds for any Hermitian matrix $W_n$. 
Let $1/10 \geq \eta \geq 1/n$ and $L,\eps,\delta > 0$.  Suppose that one has the bound
\begin{equation}\label{soda}
 |s_{n} (z) -s(z) | \le \delta
\end{equation}
with (uniformly) overwhelming probability for all $z$ with $|\Re(z)| \leq L$ and $\Im(z) \geq \eta$.  Then for any interval $I$ in $[-L+\eps,L-\eps]$ with $|I| \geq \max( 2\eta, \frac{\eta}{\delta} \log\frac{1}{\delta} )$, one has
$$ |N_I - n \int_I \rho_{sc}(x)\ dx| \ll_\eps \delta n |I|$$
with overwhelming probability, where $N_I$ is the number of eigenvalues of $W_n$ in $I$.
\end{lemma}

\begin{proof} See \cite[Lemma 60]{TVbulk}.
\end{proof}

As a consequence of this lemma (with $L=4$ and $\eps=1$, say), we see that Theorem \ref{sdb-2} follows from

\begin{theorem}[Concentration for the Stieltjes transform up to edge]\label{stieltjes}   Consider a  random Hermitian matrix $M_n = (\zeta_{ij})_{1 \leq i,j \leq n}$ whose upper-triangular entries are independent with mean zero and variance $1$, with variance $c$ on the diagonal for some absolute constant $c>0$, and such that $|\zeta_{ij}| \leq K$ almost surely for all $i,j$ and some $K \geq 1$.  Set $W_n := \frac{1}{\sqrt{n}} M_n$. 
Let $0 < \delta < 1/2$ (which can depend on $n$), and suppose that $K = o( n^{1/2} \delta^2 )$.
Then \eqref{soda} holds with (uniformly) overwhelming probability for all $z$ with $|\Re(z)| \leq 4$ and 
$$ \Im(z) \geq \frac{K^2 \log^{3.5} n}{\delta^8 n}.$$
\end{theorem}

The remainder of this section is devoted to proving Theorem \ref{stieltjes}.  Fix $z$ as in Theorem \ref{stieltjes}, thus $|\Re(z)| \leq 4$ and $\Im(z) = \eta$, where
\begin{equation}\label{etan}
 \eta n \geq \frac{K^2 \log^{3.5} n}{\delta^8}.
 \end{equation}
Our objective is to show \eqref{soda} with (uniformly) overwhelming probability.

As in previous works (in particular, \cite{ESY3,TVbulk}), the key is to exploit the fact that when $\Im z > 0$, $s(z)$ is the unique solution to the equation
\begin{equation}\label{sz}
s(z) + \frac{1}{s(z)+z} = 0
\end{equation}
with $\Im s(z) > 0$; this is immediate from \eqref{explicit}.

We now seek a similar relation for $s_n$.  Note that $\Im s_n(z) > 0$ by \eqref{sndef}.  We use the following standard matrix identity (cf. \cite[Lemma 39]{TVbulk}, or \cite[Chapter 11]{BS}):

\begin{lemma}\label{cool} We have
\begin{equation}\label{ssz}
 s_n(z) = \frac{1}{n} \sum_{k=1}^n \frac{1}{\frac{1}{\sqrt{n}} \zeta_{kk} - z - Y_k}
\end{equation}
where 
$$Y_k := a_k^* (W_{n,k} - zI)^{-1} a_k,$$
$W_{n,k}$ is the matrix $W_n$ with the $k^{th}$ row and column removed, and $a_k$ is the $k^{th}$ row of $W_n$ with the $k^{th}$ element removed.
\end{lemma}

\begin{proof} By Schur's complement, $\frac{1}{\zeta_{kk} - z - a_k^* (W_k - zI)^{-1} a_k}$ is the $k^{th}$ diagonal entry of $(W-zI)^{-1}$.  Taking traces, one obtains the claim.
\end{proof}

\begin{proposition}[Concentration of $Y_k$]\label{yk-conc}  For each $1 \leq k \leq n$, one has $Y_k = s_n(z) + o(\delta^2)$ with overwhelming probability.
\end{proposition}

\begin{proof} Fix $k$, and write $z=x+\sqrt{-1}\eta$. 

The entries of $a_k$ are independent of each other and of $W_{n,k}$, and have mean zero and variance $\frac{1}{n}$.  By linearity of expectation we thus have, on conditioning on $W_{n,k}$
$$ \E(Y_k|W_{n,k}) = \frac{1}{n} \tr (W_{n,k}-zI)^{-1} = (1 - \frac{1}{n}) s_{n,k}(z)$$
where
$$ s_{n,k}(z) := \frac{1}{n-1} \sum_{i=1}^{n-1} \frac{1}{\lambda_i(W_{n,k})-z}$$
is the Stieltjes transform of $W_{n,k}$.  From the Cauchy interlacing law \eqref{interlace} and \eqref{etan}, we have 
$$ s_n(z) - (1-\frac{1}{n}) s_{n,k}(z) = O\left( \frac{1}{n} \int_\R \frac{1}{|x-z|^2}\ dx \right) = O\left( \frac{1}{n \eta} \right) =o(\delta^2).$$
It follows that 
$$
 \E(Y_k|W_{n,k}) = s_n(z) + o(\delta^2)
$$
and so it will remain to show the concentration estimate
$$ Y_k - \E(Y_k|W_{n,k}) = o(\delta^2) $$
with overwhelming probability.  

Rewriting $Y_k$, it suffices to show that
\begin{equation}\label{rje}
\sum_{j=1}^{n-1} \frac{R_j}{\lambda_j(W_{n,k})-(x+\sqrt{-1}\eta)} = o(\delta^2) 
\end{equation}
with overwhelming probability, where $R_j :=| u_j(W_{n,k})^* a_k|^2-1/n$.

Let $1 \leq i_- < i_+ \leq n$, then
$$ \sum_{i_- \leq j \leq i_+} R_j = \| P_H a_k \|^2 - \frac{\dim(H)}{n}$$
where $H$ is the space spanned by the $u_j(W_{n,k})^*$ for $i_- \leq j \leq i_+$.  From Lemma \ref{lemma:projection} and the union bound,
 we conclude that with overwhelming probability
\begin{equation}\label{rjb-0}
 |\sum_{i_- \leq j \leq i_+} R_j| \ll \frac{\sqrt{i_+ - i_-} K \log n + K^2 \log^2 n}{n}.
\end{equation}
By the triangle inequality, this implies that
$$
 \sum_{i_- \leq j \leq i_+} \| P_H a_k \|^2  \ll \frac{i_+-i_-}{n} + \frac{\sqrt{i_+ - i_-} K \log n + K^2 \log^2 n}{n}$$
 and hence by a further application of the triangle inequality
\begin{equation}\label{rjb}
\sum_{i_- \leq j \leq i_+} |R_j| \ll \frac{(i_+ - i_-) + K^2 \log^2 n}{n}
\end{equation}
with overwhelming probability.  

The plan is to use \eqref{rjb-0} and \eqref{rjb} to establish \eqref{rje}.  Accordingly, we split the LHS of \eqref{rje}, into several subsums according to the distance 
$|\lambda_j -x|$.  Lemma \ref{lemma:smallsc} provides a sharp estimate on the number of 
terms of each subsum which will allow us to obtain a good upper bound on the absolute value. 

We turn to the details.  From \eqref{etan} we can choose two auxiliary parameters
$0 < \delta', \alpha < 1$ such that 
\begin{equation} \label{eqn:condition-delta} 
\begin{split}
\delta' &=o(\delta^2); \\
\alpha \log n &=o(\delta^2); \\
\alpha \delta' \eta n &\geq K^2 \log^2 n; \\
\frac{K \log n} {\sqrt{\alpha \delta' \eta n}} &= o(\delta^2).
\end{split} 
\end{equation} 
Indeed, one could set $\delta' := \delta^2 \log^{-0.01} n$ and $\alpha := \delta^2 \log^{-1.01} n$ and use \eqref{etan}.

Fix such parameters, and consider the contribution to \eqref{rje} of the indices $j$ for which
$$ |\lambda_j(W_n) - x |  \le \delta' \eta.$$
By Lemma \ref{lemma:smallsc} and \eqref{eqn:condition-delta}, the interval of $j$ for which this occurs has cardinality $O(\delta' \eta n) $ (with overwhelming probability).  On this interval, the quantity $\frac{1}{\lambda_j(W_{n,k})-(x+\sqrt{-1}\eta)}$ has magnitude $O(\frac{1}{\eta})$.
 Applying \eqref{rjb} (and \eqref{eqn:condition-delta}), we see that the contribution of this case is thus
$$ \ll \frac{1}{\eta} \frac{ \delta' \eta n }{n} = o(\delta^2)$$
which is acceptable.

Next, we consider the contribution to \eqref{rje} of those indices $j$ for which
$$ (1+\alpha)^l \delta' \eta < |\lambda_j(W_n) - x | \leq (1+\alpha)^{l+1} \delta' \eta$$
for some integer $0 \leq l \ll \log n/\alpha$, and then sum over $l$.  By Lemma \ref{lemma:smallsc} and \eqref{eqn:condition-delta}, the set of $j$ for which this occurs is contained (with overwhelming probability) in at most two intervals of cardinality $O( (1+\alpha)^l \alpha \delta' \eta n )$.  On each of these intervals, the quantity $\frac{1}{\lambda_j(W_{n,k})-(x+\sqrt{-1}\eta)}$ has magnitude $O\left( \frac{1}{(1+\alpha)^l \delta' \eta} \right)$ and fluctuates by $O\left( \frac{\alpha}{(1+\alpha)^l \delta' \eta} \right)$.  Applying \eqref{rjb-0}, \eqref{rjb} (and noting that $(1+\alpha)^l \alpha \delta' \eta n $ exceeds $K^2 \log^2 n$, by \eqref{eqn:condition-delta}) we see that the contribution of a single $l$ to \eqref{rje} is at most
$$ \ll \frac{1}{(1+\alpha)^l \delta' \eta} \frac{ \sqrt{\alpha (1+\alpha)^l \delta' \eta n} K \log n }{n} + 
\frac{\alpha}{(1+\alpha)^l \delta' \eta} \frac{ \alpha (1+\alpha)^l \delta' \eta n }{n}$$
which simplifies to
$$ \ll \alpha (1+\alpha)^{-l/2} \frac{K \log n}{\sqrt{\alpha \delta' \eta n}} + \alpha^2.$$
Summing over $l$ we obtain a bound of
$$ \ll \frac{K \log n}{\sqrt{\alpha \delta' \eta n}} + \alpha \log n$$
which is acceptable by \eqref{eqn:condition-delta}.
\end{proof}

We now conclude the proof of Theorem \ref{sdb-2}.    By hypothesis, 
$$|\frac{1}{\sqrt{n}} \zeta_{kk}| \leq K/\sqrt{n} = o( \delta^2 )$$
almost surely.  Inserting these bounds into \eqref{ssz}, we see that
with overwhelming probability 
$$
s_n(z) + \frac{1}{n} \sum_{k=1}^n \frac{1}{s_n(z) + z + o(\delta^2)} = 0.$$
By the triangle inequality (and square rooting the $o()$ decay), we can assume that either the error term $o(\delta^2)$ is $o( \delta^2 |s_n(z)+z| )$, or that $|s_n(z)+z|$ is $o(1)$.  Suppose the former holds.  Then by Taylor expansion
$$ \frac{1}{s_n(z) + z + o(\delta^2)} = \frac{1}{s_n(z)+z} + o( \delta^2 ) $$
and thus
$$ s_n(z) + \frac{1}{s_n(z)+z} = o( \delta^2 ).$$
If we assume $|z| \leq 10$ (say), we conclude that $|s_n(z)| \leq 100$.  Multiplying out by $s_n(z)+z$ and rearranging, we obtain
$$ (s_n(z) + \frac{z}{2})^2 = \frac{z^2-4}{4} + o(\delta^2).$$
Thus
$$ s_n(z) + \frac{z}{2} = \pm \sqrt{\frac{z^2-4}{4}} + o(\delta)$$
(treating the case when $z^2-4=o(\delta^2)$ separately).  

To summarise, we have shown (with overwhelming probability) that in the region
$$ |z| \leq 10; \quad |\Re(z)| \leq 4; \quad \Im(z) \geq \frac{K^2 \log^{3.5} n}{\delta^8 n}$$
that one either has $s_n(z) = s(z) + o(\delta)$, $s_n(z) = - z - s(z) + o(1) = s(z) - \sqrt{z^2-4} + o(1)$, or $|s_n(z)+z| = o(1)$.   It is not hard to see that the first two cases are disconnected from the third (for $n$ large enough) in this region, because $s(z)=\frac{-1}{s(z)+z}$ is bounded away from zero, as is $s(z)+z = \frac{-1}{s(z)}$.  Furthermore, the first and second possibilities are also disconnected from each other except when $z^2-4 = o(\delta^2)$.  Also, the second and third possibilities can only hold for $\Im(z)=o(1)$ since $s_n(z)$ and $z$ both have positive real part.  A continuity argument thus shows that the first possibility must hold throughout the region except when $z^2-4 = o(\delta^2)$, in which case either the first or second possibility can hold; but in that region, the first and second possibility are equivalent, and \eqref{soda} follows.  The proof of Theorem \ref{sdb-2} is now complete.

\section{Delocalization of eigenvectors}\label{deloc-sec}

Without loss of generalization, we can assume that the entries are continuously distributed. 
 Having established Theorem \ref{sdb-2}, we now use this theorem to establish Proposition \ref{deloc2}.

Let $M_n$ obey condition {\bf C0}.  Then by Markov's inequality, one has $|\zeta_{ij}| \ll \log^{O(1)} n$ with overwhelming probability (here and in the sequel we allow implied constants in the $O()$ notation to depend on the constants $C,C'$ in \eqref{ued}).  By conditioning the $\zeta_{ij}$ to this event\footnote{Strictly speaking, this distorts the mean and variance of $\zeta_{ij}$ by an exponentially small amount, but one can easily check that this does not significantly impact any of the arguments in this section.}, we may thus assume that 
\begin{equation}\label{zak}
|\zeta_{ij}| \leq K
\end{equation}
almost surely for some $K = O( \log^{O(1)} n)$.

Fix $1 \leq i \leq n$; by symmetry we may take $i \geq n/2$.  By the union bound and another application of symmetry, it suffices to show that
$$ |u_i(M_n)^* e_1| \ll n^{-1/2} \log^{O(1)} n$$
with overwhelming probability.

To compute $u_i(M_n)^* e_1$ we use the following identity from \cite{ESY1} (see also \cite[Lemma 38]{TVbulk}):

\begin{lemma} \label{lemma:firstcoordinate} Let 
$$ A_n = \begin{pmatrix} a & X^* \\ X & A_{n-1} \end{pmatrix}$$
be a $n \times n$ Hermitian matrix for some $a \in \R$ and $X \in \C^{n-1}$, and let $\begin{pmatrix} x \\ v \end{pmatrix}$ be a unit eigenvector of $A$ with eigenvalue $\lambda_i(A)$, where $x \in \C$ and $v \in \C^{n-1}$.  Suppose that none of the eigenvalues of $A_{n-1}$ are equal to $\lambda_i(A)$.  Then
$$|x|^2 =  \frac{1}{1 + \sum_{j=1}^{n-1}  (\lambda_j(A_{n-1})-\lambda_i(A_n))^{-2} |u_j(A_{n-1})^* X|^2}, $$ 
where $u_j(A_{n-1})$ is a unit eigenvector corresponding to the eigenvalue $\lambda_j(A_{n-1})$.
\end{lemma}

\begin{proof}  By subtracting $\lambda_i(A) I$ from $A$ we may assume $\lambda_i(A) =0$.  The eigenvector equation then gives
$$ x X + A_{n-1} v' = 0,$$
thus
$$ v' = - x A_{n-1}^{-1} X.$$
Since $\|v'\|^2 + |x|^2 = 1$, we conclude
$$ |x|^2 (1 + \|A_{n-1}^{-1} X\|^2 ) = 1.$$
Since $\|A_{n-1}^{-1} X\|^2 = \sum_{j=1}^{n-1}  (\lambda_j(A_{n-1}))^{-2} |u_j(A_{n-1})^* X|^2$, the claim follows.
\end{proof}

Let $M_{n-1}$ be the bottom right $n-1 \times n-1$ minor of $M_n$. As we are assuming that the coefficients of $M_n$ are continuously distributed, we see almost surely that none of the eigenvalues of $M_{n-1}$ are equal to $\lambda_i(M_n)$.  We may thus apply Lemma \ref{lemma:firstcoordinate} and conclude that
$$ |u_i(M_n)^* e_1|^2 =  \frac{1}{1 + \sum_{j=1}^{n-1}   \frac{|u_j(M_{n-1})^* X|^2}{(\lambda_j(M_{n-1})-\lambda_i(M_n))^{2}}}$$
where $X$ is the bottom left $n-1 \times 1$ vector of $M_n$ (and thus has entries $\zeta_{j1}$ for $1 < j \leq n$).  It thus suffices to show that
$$ \sum_{j=1}^{n-1}  \frac{|u_j(M_{n-1})^* X|^2}{(\lambda_j(M_{n-1})-\lambda_i(M_n))^{2}} \gg n \log^{-O(1)} n$$
with overwhelming probability.

It will be convenient to eliminate the exponent $2$ in the denominator, as follows.
From Lemma \ref{lemma:projection}, one has $|u_j(M_{n-1})^* X| \ll \log^{O(1)} n$ with overwhelming probability for each $j$ (and hence for all $j$, by the union bound).  It thus suffices to show that
$$ \sum_{j=1}^{n-1}  \frac{|u_j(M_{n-1})^* X|^4}{(\lambda_j(M_{n-1})-\lambda_i(M_n))^{2}} \gg n \log^{-O(1)} n$$
with overwhelming probability.  By the Cauchy-Schwarz inequality, it thus suffices to show that
$$
 \sum_{j: i-T_- \leq j \leq i+T_+}  \frac{|u_j(M_{n-1})^* X|^2}{|\lambda_j(M_{n-1})-\lambda_i(M_n)|} \gg n^{1/2} \log^{-O(1)} n
$$
with overwhelming probability for some $1 \leq T_-, T_+ \ll \log^{O(1)} n$.  It is convenient to work with the normalized matrix $W_n := \frac{1}{\sqrt{n}} M_n$, thus we need to show 
\begin{equation}\label{simo}
 \sum_{j: i-T_- \leq j \leq i+T_+}  \frac{|u_j(W_{n-1})^* Y|^2}{|\lambda_j(W_{n-1})-\lambda_i(W_n)|} \gg \log^{-O(1)} n
\end{equation}
with overwhelming probability for some $1 \leq T_-, T_+ \ll \log^{O(1)} n$, where $Y := \frac{1}{\sqrt{n}} X$ has entries $\frac{1}{\sqrt{n}} \zeta_{j1}$ for $1 < j \leq n$.

There are two cases: the bulk case and the edge case; the former was already treated in \cite{TVbulk}, but the latter is new.

\subsection{The bulk case}

Suppose that $n/2 \leq i < 0.999 n$.  
Then from the semicircular law (or Theorem \ref{sdb-2}) we see that $\lambda_i(W_n) \in [-2+\eps,2+\eps]$ with overwhelming probability for some absolute constant $\eps > 0$.  Let $A$ be a large constant to be chosen later.  A further application of Theorem \ref{sdb-2} then shows that there is an interval $I$ of length $\log^{A} n / n$ centered at $\lambda_i(W_n)$ which contains $\Theta( \log^{A} n )$ eigenvalues of $W_n$.  If $\lambda_j(W_n), \lambda_{j+1}(W_n)$ lie in $I$, then by the Cauchy interlacing property \eqref{interlace}, $|\lambda_j(W_{n-1})-\lambda_i(W_n)| \ll \log^{A} n/ n$.
One can thus lower bound the left-hand side of \eqref{simo} (for suitable values of $T$) by
$$ \gg n \log^{-A} n \sum_{j: \lambda_j(W_n), \lambda_{j+1}(W_n) \in I} |u_j(W_{n-1})^* Y|^2.$$
One can rewrite this as $\log^{-A} n \| \pi_H X \|^2$, where $H$ is the span of the $u_j(W_{n-1})$ for $\lambda_j(W_n), \lambda_{j+1}(W_n) \in I$.  The claim then follows from Lemma \ref{lemma:projection} (for $A$ large enough).

\subsection{The edge case}

We now turn to the more interesting edge case when $0.999 n \leq i \leq n$.  Using the semicircular law, we now see that
\begin{equation}\label{wj}
 \lambda_i(W_n) \geq 1.9
\end{equation}
(say) with overwhelming probability.

Next, we can exploit the following identity:

\begin{lemma}[Interlacing identity]\label{jn-lem}\cite[Lemma 37]{TVbulk}  If $u_j(W_{n-1})^* X$ is non-zero for every $j$, then
\begin{equation}\label{jn}
 \sum_{j=1}^{n-1} \frac{|u_j(W_{n-1})^* X|^2}{\lambda_j(W_{n-1}) - \lambda_i(W_n)} = \frac{1}{\sqrt{n}} \zeta_{nn} - \lambda_i(W_n).
\end{equation}
\end{lemma}

\begin{proof} By diagonalising $W_{n-1}$ (noting that this does not affect either side of \eqref{jn}), we may assume that $W_{n-1} = \operatorname{diag}(\lambda_1(W_{n-1}),\ldots,\lambda_{n-1}(W_{n-1}))$ and $u_j(W_{n-1}) = e_j$ for $j=1,\ldots,n-1$.  One then easily verifies that the characteristic polynomial $\det(W_n - \lambda I)$ of $W_n$ is equal to
$$
\prod_{j=1}^{n-1} (\lambda_j(W_{n-1}) - \lambda) [ (\frac{1}{\sqrt{n}} \zeta_{nn} - \lambda) - \sum_{j=1}^{n-1} \frac{|u_j(W_{n-1})^* X|^2}{\lambda_j(W_{n-1}) - \lambda} ]
$$
when $\lambda$ is distinct from $\lambda_1(W_{n-1}),\ldots,\lambda_{n-1}(W_{n-1})$.  Since $u_j(W_{n-1})^* X$ is non-zero by hypothesis, we see that this polynomial does not vanish at any of the $\lambda_j(W_{n-1})$.  Substituting $\lambda_i(W_n)$ for $\lambda$, we obtain \eqref{jn}.  
\end{proof}

Again, the continuity of the entries of $M_n$ ensure that the hypothesis of Lemma \ref{jn-lem} is obeyed almost surely.  From \eqref{zak}, \eqref{wj}, \eqref{jn} one has
$$
|\sum_{j=1}^{n-1} \frac{|u_j(W_{n-1})^* X|^2}{\lambda_j(W_{n-1}) - \lambda_i(W_n)}| \geq 1.9 - o(1)$$
with overwhelming probability, so to show \eqref{simo}, it will suffice by the triangle inequality to show that
\begin{equation}\label{jiggle}
 |\sum_{j> i+T_+ \hbox{ or } j < i-T_-} \frac{|u_j(W_{n-1})^* X|^2}{\lambda_j(W_{n-1}) - \lambda_i(W_n)}| \leq 1.8 + o(1)
\end{equation}
(say) with overwhelming probability for some $T = \log^{O(1)} n$.

Let $A > 100$ be a large constant to be chosen later.  By Theorem \ref{sdb-2}, we see (if $A$ is large enough) that
\begin{equation}\label{sor}
N_I = n \alpha_I |I| + O( |I| n \log^{-A/20} n ) 
\end{equation}
with overwhelming probability for any interval $I$ of length $|I|=\log^A n/n$, where $\alpha_I := \frac{1}{|I|} \int_I \rho_{sc}(x)\ dx$.  For any such interval, we see from Lemma \ref{lemma:projection} (and Cauchy interlacing \eqref{interlace}) that with overwhelming probability
$$ \sum_{j: \lambda_j(W_{n-1}) \in I} |u_j(W_{n-1})^* X|^2 = \frac{N_I}{n} + O\left( \frac{\log^{A/2 + O(1)} n}{n} \right)$$
and thus by \eqref{sor} (for $A$ large enough)
$$ \sum_{j: \lambda_j(W_{n-1}) \in I} |u_j(W_{n-1})^* X|^2 = \alpha_I |I| + O( |I| \log^{-A/20)} n  ).$$
Set $d_I := \frac{\dist( \lambda_i(W_n), I )}{|I|}$.  If $d_I \geq \log n$ (say), then 
$$\frac{1}{\lambda_j(W_{n-1}) - \lambda_i(W_n)} = \frac{1}{d_I |I|} + O\left( \frac{1}{d_I^2 |I|} \right)$$
for all $j$ in the above sum, thus
\begin{equation}\label{wii}
 \sum_{j: \lambda_j(W_{n-1}) \in I} \frac{|u_j(W_{n-1})^* X|^2}{\lambda_j(W_{n-1}) - \lambda_i(W_n)} = \frac{\alpha_I}{d_I}
 + O\left( \frac{\log^{-A/20} n}{d_I} \right) + O\left( \frac{\alpha_I}{d_I^2} \right) .
\end{equation}
We now partition the real line into intervals $I$ of length $\log^A n/n$, and sum \eqref{wii} over all $I$ with $d_I \geq \log n$.  Bounding $\alpha_I$ crudely by $O(1)$, we see that $\sum_I O\left( \frac{\alpha_I}{d_I^2} \right)  = O\left( \frac{1}{\log n} \right) = o(1)$.  Similarly, one has
$$ \sum_I O\left( \frac{\log^{-A/20} n}{d_I} \right)  = O( \log^{-A/20} n \log n ) = o(1)$$
if $A$ is large enough.  Finally, Riemann integration of the principal value integral 
$$ p.v. \int_{-2}^2 \frac{\rho_{sc}(x)}{x - \lambda_i(W_n)}\ dx  := \lim_{\eps \to 0} \int_{|x| \leq 2: |x-\lambda_i(W_n)| > \eps} \frac{\rho_{sc}(x)}{x - \lambda_i(W_n)}\ dx  $$
shows that
$$ \sum_I \frac{\alpha_I}{d_I} = p.v. \int_{-2}^2 \frac{\rho_{sc}(x)}{x - \lambda_i(W_n)}\ dx + o(1).$$
The operator norm of $W_n$ is $2+o(1)$ with overwhelming probability (see e.g. \cite{BS}, \cite{baiyin2}), so $|\lambda_i(W_n)| \leq 2+o(1)$.  Using the formula \eqref{explicit} for the Stieltjes transform, one obtains from residue calculus that
$$ p.v. \int_{-2}^2 \frac{\rho_{sc}(x)}{x - \lambda_i(W_n)}\ dx = - \lambda_i(W_n) / 2$$
for $|\lambda_i(W_n)| \leq 2$, with the right-hand side replaced by $- \lambda_i(W_n) / 2 + \sqrt{\lambda_i(W_n)^2-4}/2$ for $|\lambda_i(W_n)| > 2$.
In either event, we have
$$ |p.v. \int_{-2}^2 \frac{\rho_{sc}(x)}{x - \lambda_i(W_n)}\ dx| \leq 1+o(1).$$
Putting all this together, we see that
$$
| \sum_{I: d_I \geq \log n} \sum_{j: \lambda_j(W_{n-1}) \in I} \frac{|u_j(W_{n-1})^* X|^2}{\lambda_j(W_{n-1}) - \lambda_i(W_n)} | \leq 1+o(1).
$$
The intervals $I$ with $d_I < \log n$ will contribute at most $\log^{A+O(1)} n$ eigenvalues, by \eqref{sor} (and Cauchy interlacing \eqref{interlace}).  The claim \eqref{jiggle} now follows by setting $T_-$ and $T_+$ appropriately.  The proof of Proposition \ref{deloc2} is now complete.

\begin{remark} From \eqref{simo} and Lemma \ref{lemma:projection} one sees that
$$ |\lambda_{i-1}(W_{n-1})-\lambda_i(W_n)| \ll \log^{O(1)} n / n$$
with overwhelming probability for all $n/2 \leq i \leq n$, and similarly one has
$$ |\lambda_{i}(W_{n-1})-\lambda_i(W_n)| \ll \log^{O(1)} n / n$$
with overwhelming probability for all $1 \leq i \leq n/2$.  On the other hand, according to the Tracy-Widom law, the gap between $\lambda_n(W_n)$ and $\lambda_{n-1}(W_n)$ (or between $\lambda_1(W_n)$ and $\lambda_2(W_n)$) can be expected to be as large as $n^{-2/3}$.  Thus we see that there is a significant bias at the edge in the interlacing law \eqref{interlace}, which can ultimately be traced to the imbalance of ``forces'' in the interlacing identity \eqref{jn} at that edge.
\end{remark}

\section{Lower bound on eigenvalue gap}\label{ltail-sec}

We now give the proof of Theorem \ref{ltail2}.  Most of the proof will follow closely the proof of Theorem \ref{ltail} in \cite{TVbulk}, so we shall focus on the changes needed to that argument.  As such, this section will assume substantial familiarity with the material from \cite{TVbulk}, and will cite from it repeatedly.  Similarly for the next section.

For technical reasons relating to an induction argument, it turns out that one has to treat the extreme cases $i=1,n$ separately:

\begin{proposition}[Extreme cases]\label{extremes}  Theorem \ref{ltail2} is true when $i=1$ or $i=n$.
\end{proposition}

\begin{proof} By symmetry it suffices to do this for $i=n$.  By a limiting argument we may assume that the entries $\zeta_{ij}$ of $M_n$ are continuously distributed.  From Lemma \ref{jn-lem} one has (almost surely) that
$$
 \sum_{j=1}^{n-1} \frac{|u_j(W_{n-1})^* X|^2}{\lambda_j(W_{n-1}) - \lambda_n(W_n)} = \frac{1}{\sqrt{n}} \zeta_{nn} - \lambda_n(W_n).
$$
Recall that $\lambda_n(W_n)=2+o(1)$ with overwhelming probability; also, $\frac{1}{\sqrt{n}} \zeta_{nn}=o(1)$ with overwhelming probability.  As all the terms in the left-hand side have the same sign, we conclude that
$$ \frac{|u_{n-1}(W_{n-1})^* X|^2}{|\lambda_{n-1}(W_{n-1}) - \lambda_n(W_n)|} \ll 1.$$
From Theorem \ref{bes} and Proposition \ref{deloc2}, we have $|u_{n-1}(W_{n-1})^* X| \geq n^{-c_0/10}$ (say) with high probability, and so
$$|\lambda_{n-1}(W_{n-1}) - \lambda_n(W_n)| \geq n^{-c_0}$$
with high probability.  The claim now follows from the Cauchy interlacing property \eqref{interlace}.
\end{proof}

\begin{remark}  In fact, at the edge, one should be able to improve the lower bound on the eigenvalue gap substantially, from $n^{-c_0}$ to $n^{1/3-c_0}$, in accordance to the Tracy-Widom law, but we will not need to do so here.
\end{remark}

Now we handle the general case of Theorem \ref{ltail2}.  Fix $M_n$ and $c_0$.  We write $n_0, i_0$ for $i,n$, thus $1 \leq i_0 \leq n_0$ and our task is to show that
$$ \lambda_{i_0+1}(A_n) - \lambda_{i_0}(A_{n_0}) \ge n^{-c_0}$$
with high probability.  By Proposition \ref{extremes} we may assume $1 < i_0 < n_0$.  We may also assume $n_0$ to be large, as the claim is vacuous otherwise.  As in previous sections, we may truncate so that all coefficients $\zeta_{ij}$ are of size $O( \log^{O(1)} n_0 )$ (as before, the exponentially small corrections to the mean and variance of $\zeta_{ij}$ caused by this are easily controlled), and approximate so that the distribution is continuous rather than discrete.

For each $n_0/2 \leq n \leq n_0$, let $A_n$ be the top left $n \times n$ minor of $A_{n_0}$.  As in \cite[Section 3.4]{TVbulk}, we introduce the \emph{regularized gap}
\begin{equation}\label{giln}
g_{i,l,n} := \inf_{1 \leq i_- \leq i-l < i \leq i_+ \leq n} \frac{\lambda_{i_+}(A_n)-\lambda_{i_-}(A_n)}{\min( i_+-i_-, \log^{C_1} n_0 )^{\log^{0.9} n_0}},
\end{equation}
for all $n_0/2 \leq n \leq n_0$ and $1 \leq i-l < i \leq n$, where $C_1$ is a large constant to be chosen later.  It will suffice to show that for each $1 < i_0 < n_0$, that
$$ g_{i_0,1,n_0} \leq n^{-c_0}$$
with high probability.  By symmetry we may assume that $n_0/2 \leq i_0 < n_0$.

As before, we let $u_1(A_n),\ldots,u_n(A_n)$ be an orthonormal eigenbasis of $A_n$ associated to the eigenvectors $\lambda_1(A_n),\ldots,\lambda_n(A_n)$.  We also let $X_n \in \C^n$ be the rightmost column of $A_{n+1}$ with the bottom coordinate $\sqrt{n} \zeta_{n+1,n+1}$ removed.

We will need two key lemmas.  First, we have the following deterministic lemma (a minor variant of \cite[Lemma 47]{TVbulk}), showing that a narrow gap can be propagated backwards in $n$ unless one of a small number of exceptional events happen:

\begin{lemma}[Backwards propagation of gap]\label{backprop}  Suppose that $i_0 \leq n+1 \leq n_0$ and $l \leq n/10$ is such that
\begin{equation}\label{gdel}
g_{i_0,l,n+1} \leq \delta
\end{equation}
for some $0 < \delta \leq 1$ (which can depend on $n$), and that
\begin{equation}\label{pill}
 \lambda_{n+1}(A_{n+1})-\lambda_{n}(A_{n+1}) \geq \delta \exp( \log^{0.91} n_0 ).
\end{equation}
Then $i_0 \leq n$.  Suppose further that
\begin{equation}\label{gilp}
 g_{i_0,l+1,n} \geq 2^m g_{i_0,l,n+1}
\end{equation}
for some $m \geq 0$ with
\begin{equation}\label{mcivil}
2^m \leq \delta^{-1/2}.
\end{equation}
Then one of the following statements hold:
\begin{itemize}
\item[(i)]  (Macroscopic spectral concentration) There exists $1 \leq i_- < i_+ \leq n+1$ with $i_+-i_- \geq \log^{C_1/2} n$ such that $|\lambda_{i_+}(A_{n+1}) - \lambda_{i_-}(A_{n+1})| \leq \delta \exp( \log^{0.95} n ) (i_+-i_-)$.
\item[(ii)]  (Small inner products)  There exists $1 \leq i_- \leq i_0-l < i_0 \leq i_+ \leq n$ with $i_+-i_- \leq \log^{C_1/2} n$ such that
\begin{equation}\label{smallin}
\sum_{i_- \leq j < i_+} |u_j(A_n)^* X_n|^2 \leq \frac{n (i_+-i_-)}{2^{m/2} \log^{0.01} n}
\end{equation}
\item[(iii)]  (Large coefficient)  We have
$$ |\zeta_{n+1,n+1}| \geq n^{0.4}.$$
\item[(iv)] (Large eigenvalue)  For some $1 \leq i \leq n+1$ one has
$$ |\lambda_i(A_{n+1})| \geq \frac{n \exp( -\log^{0.95} n )}{\delta^{1/2}}.$$
\item[(v)] (Large inner product)  There exists $1 \leq i \leq n$ such that
$$ |u_i(A_n)^* X_n|^2 \geq \frac{n \exp( - \log^{0.96} n )}{\delta^{1/2}}.$$
\item[(vi)] (Large row)  We have
$$ \|X_n\|^2 \geq \frac{n^2 \exp( - \log^{0.96} n )}{\delta^{1/2}}.$$
\item[(vii)] (Large inner product near $i_0$)  There exists $1 \leq i \leq n$ with $|i-i_0| \leq \log^{C_1} n$ such that
$$ |u_i(A_n)^* X_n|^2 \geq 2^{m/2} n \log^{0.8} n.$$
\end{itemize}
\end{lemma}

\begin{proof}
The proof of this proposition repeats the proof of \cite[Lemma 47]{TVbulk} in \cite[Section 6]{TVbulk} almost exactly.  Only the following changes have to be made:

\begin{itemize}
\item We have the upper bound $\lambda_{i_+}(A_{n+1}) - \lambda_{i_-}(A_{n+1}) \leq \delta (\log^{C_1} n)^{\log^{0.9} n_0}$, which together with \eqref{pill} forces $i_+ \neq n+1$ and thus $i_0 \leq n$ as required.
\item The variable $j$ now lies in the range $1 \leq j \leq n$ rather than $\eps n/10 \leq j \leq (1-\eps/10) n$.
\item $i_{--}$ has to be defined as $\max(i_--2^{k-1},1)$ rather than just $i_- - 2^{k-1}$ (and similarly for $i_{++}$).
\end{itemize}
\end{proof}

Next, we need the following result that asserts that the events (i)-(vii) are rare:

\begin{proposition}[Bad events are rare]\label{bad-event}  Suppose that $n_0/2 \leq n < n_0$ and $l \leq n/10$, and set $\delta := n_0^{-\kappa}$ for some sufficiently small fixed $\kappa > 0$.  Let $m \geq 0$ be such that $2^m \leq \delta^{-1/2}$. Then:
\begin{itemize}
\item[(a)] The events (i), (iii), (iv), (v), (vi) in Lemma \ref{backprop} all fail with high probability.  
\item[(b)] There is a constant $C'$ such that all the coefficients of the eigenvectors $u_j(A_n)$ for $1 \leq j \leq n$ are of magnitude at most $n^{-1/2} \log^{C'} n$ with overwhelming probability.  Conditioning $A_n$ to be a matrix with this property, the events (ii) and (vii) occur with a conditional probability of at most $2^{-\kappa m} + n^{-\kappa}$.  
\item[(c)] Furthermore, there is a constant $C_2$ (depending on $C',\kappa,C_1$) such that if $l \geq C_2$ and $A_n$ is conditioned as in (b), then (ii) and (vii) in fact occur with a conditional probability of at most $2^{-\kappa m} \log^{-2C_1} n + n^{-\kappa}$.
\end{itemize}
\end{proposition}

\begin{proof}
The proof of this proposition repeats the proof of \cite[Proposition 49]{TVbulk} in \cite[Section 7]{TVbulk} almost exactly.  Only the following changes have to be made:

\begin{itemize}
\item All references to \cite[Theorem 56]{TVbulk} (i.e. Theorem \ref{sdb}) need to be replaced with Theorem \ref{sdb-2}.
\item The variable $j$ now lies in the range $1 \leq j \leq n$ rather than $\eps n/2 \leq j \leq (1-\eps/2) n$.
\end{itemize}
\end{proof}

Given Lemma \ref{backprop} and Proposition \ref{bad-event}, the proof of Theorem \ref{ltail2} exactly follows the proof of Theorem \ref{ltail} in \cite[Section 3.5]{TVbulk}, with the following minor changes:
\begin{itemize}
\item In the definition of the event $E_n$, the range $\eps n/2 \leq j \leq (1-\eps/2) n$ needs to be expanded to $1 \leq j \leq n$.
\item In the definition of the event $E_0$, the events that \eqref{pill} fail for some $n_0 - \log^{2C_1} n_0 \leq n \leq n_0$ have to be included; but these events occur with polynomially small probability, thanks to Proposition \ref{extremes} and the union bound.
\end{itemize}

This concludes the proof of Theorem \ref{ltail2}.

\section{The four moment theorem}\label{fmt-sec}

We now prove Theorem \ref{theorem:main2}.  As in \cite[Section 3.3]{TVbulk}, the proof is based on two key propositions.
 The first proposition asserts that one can swap a single coefficient (or more precisely, two coefficients) of a (deterministic) matrix $A$ as long as $A$ obeys a certain ``good configuration condition'':

\begin{proposition}[Replacement given a good configuration]\label{swap}  There exists a positive constant $C_1$ such that the following holds.  Let $k \geq 1$ and $\eps_1 > 0$, and assume $n$ sufficiently large depending on these parameters.  Let $1 \leq i_1 < \ldots < i_k \leq n$.
For a complex parameter $z$, let $A(z)$ be a (deterministic) family of $n \times n$ Hermitian matrices of the form
$$ A(z) = A(0) + z e_p e_q^* + \overline{z} e_q e_p^*$$
where $e_p, e_q$ are unit vectors.  We assume that for every $1 \leq j \leq k$ and every $|z| \leq n^{1/2+\eps_1}$ whose real and imaginary parts are multiples of $n^{-C_1}$, we have
\begin{itemize}
\item (Eigenvalue separation)  For any $1 \leq i \leq n$ with $|i-i_j| \geq n^{\eps_1}$, we have
\begin{equation}\label{noon}
 |\lambda_i(A(z)) - \lambda_{i_j}(A(z))| \geq n^{-\eps_1} |i-i_j|.
\end{equation}
\item (Delocalization at $i_j$)  If $P_{i_j}(A(z))$ is the orthogonal projection to the eigenspace associated to $\lambda_{i_j}(A(z))$, then
\begin{equation}\label{pz1}
 \| P_{i_j}(A(z)) e_p \|, \| P_{i_j}(A(z)) e_q \| \leq n^{-1/2+\eps_1}.
\end{equation}
\item For every $\alpha \geq 0$
\begin{equation}\label{pz2}
\| P_{i_j,\alpha}(A(z)) e_p \|, \| P_{i_j,\alpha}(A(z)) e_q \| \leq 2^{\alpha/2} n^{-1/2+\eps_1},
\end{equation}
whenever $P_{i_j,\alpha}$ is the orthogonal projection to the eigenspaces corresponding to eigenvalues $\lambda_i(A(z))$ with $2^\alpha \leq |i-i_j| < 2^{\alpha+1}$.
\end{itemize}
We say that $A(0), e_p, e_q$ are a \emph{good configuration} for $i_1,\ldots,i_k$ if the above properties hold.  Assuming this good configuration, then we have
\begin{equation}\label{barg}
\E (F(\zeta)) = \E F(\zeta') + O( n^{-(r+1)/2 + O(\eps_1)} ), 
\end{equation}
whenever
$$
F(z) := G( \lambda_{i_1}(A(z)),\ldots,\lambda_{i_k}(A(z)),Q_{i_1}(A(z)),\ldots,Q_{i_k}(A(z))),$$
and
$$ G = G( \lambda_{i_1},\ldots,\lambda_{i_k}, Q_{i_1},\ldots,Q_{i_k} )$$
is a smooth function from $\R^k \times \R_+^k \to \R$ that is supported on the region
$$ Q_{i_1},\ldots,Q_{i_k} \leq n^{\eps_1}$$
and obeys the derivative bounds
$$ |\nabla^j G| \leq n^{\eps_1}$$
for all $0 \leq j \leq 5$, and $\zeta, \zeta'$ are random variables with $|\zeta|, |\zeta'| \leq n^{1/2+\eps_1}$ almost surely, which 
match to order $r$ for some $r=2,3,4$.  

If $G$ obeys the improved derivative bounds
$$ |\nabla^j G| \leq n^{-C j \eps_1}$$
for $0 \leq j \leq 5$ and some sufficiently large absolute constant $C$, then we can strengthen $n^{-(r+1)/2 + O(\eps_1)}$ in \eqref{barg} to $ n^{-(r+1)/2 - \eps_1}$.
\end{proposition}

\begin{proof} See \cite[Proposition 43]{TVbulk}.
\end{proof}

The second proposition asserts that these good configurations occur very frequently:

\begin{proposition}[Good configurations occur very frequently]\label{lemma:GCC}
Let $\eps_1 > 0$ and $C, C_1, k \geq 1$.  Let $1 \leq i_1 < \ldots < i_k \leq n$, let $1 \leq p,q \leq n$, let $e_1,\ldots,e_n$ be the standard basis of $\C^n$, and let $A(0) = (\zeta_{ij})_{1 \leq i,j \leq n}$ be a random Hermitian matrix with independent upper-triangular entries and $|\zeta_{ij}| \leq n^{1/2} \log^C n$ for all $1 \leq i,j \leq n$, with $\zeta_{pq}=\zeta_{qp}=0$, but with $\zeta_{ij}$ having mean zero and variance $1$ for all other $ij$, except on the diagonal where the variance is instead $c$ for some absolute constant $c>0$, and also being distributed continuously in the complex plane.  Then $A(0),e_p,e_q$ obey the Good Configuration Condition in Theorem \ref{swap} for $i_1,\ldots,i_k$ and with the indicated value of $\eps_1, C_1$ with overwhelming probability.
\end{proposition}

\begin{proof}
The proof of this proposition repeats the proof of \cite[Proposition 44]{TVbulk} in \cite[Section 5]{TVbulk} almost exactly.  Only the following changes have to be made:

\begin{itemize}
\item All references to \cite[Theorem 56]{TVbulk} (i.e. Theorem \ref{sdb}) need to be replaced with Theorem \ref{sdb-2}.
\item All references to \cite[Proposition 58]{TVbulk} (i.e. Proposition \ref{deloc}) need to be replaced with Proposition \ref{deloc2}.
\item The edge regions in which $\lambda_i(A(z))$ do not fall inside the bulk region $[(-2+\eps')n, (2-\eps') n]$ no longer need to be treated separately, thus simplifying the last paragraph of the proof somewhat.
\end{itemize}
\end{proof}

Given these two propositions, the proof of Theorem \ref{theorem:main2} repeats the proof of \cite[Theorem 15]{TVbulk} in \cite[Section 3.3]{TVbulk} almost exactly.  Only the following changes have to be made:

\begin{itemize}
\item All references to \cite[Proposition 44]{TVbulk} need to be replaced with Proposition \ref{lemma:GCC}.
\end{itemize}

The proof of Theorem \ref{theorem:main2} is now complete.

\end{document}